\documentclass{amsart}
\usepackage{amsmath,amsfonts,amssymb,amsthm}
\usepackage{url}
\usepackage{graphicx,color}
\usepackage{enumerate}

\newtheorem{theorem}{Theorem}[section]
\newtheorem{lemma}[theorem]{Lemma}

\newtheorem{prop}[theorem]{Proposition}

\newtheorem*{theorem*}{Theorem}

\theoremstyle{definition}
\newtheorem{definition}{Definition}[section]

\theoremstyle{remark}

\numberwithin{equation}{section}

\author{Ioana-Claudia Laz\u{a}r}
\address{
Dept. of Mathematics\\
Politehnica University of Timi\c{s}oara\\
Victoriei Square $2$, $300006$-Timi\c{s}oara, Romania}
\email{ioana.lazar@upt.ro}


\begin{document}

\title[A combinatorial negative curvature condition]{A combinatorial negative curvature condition implying Gromov hyperbolicity}

\begin{abstract}
We explore a local combinatorial condition on a simplicial complex, called $8$-location. We show that $8$-location and simple connectivity imply Gromov hyperbolicity.
\end{abstract}

\maketitle

\section{Introduction}

Curvature can be expressed both in metric and combinatorial terms. Metrically, one can refer to ’nonpositively
curved’ (respectively, ’negatively curved’) metric spaces in the sense of Aleksandrov, i.e. by comparing small triangles in the space with
triangles in the Euclidean plane (hyperbolic plane). These are the CAT(0) (respectively, CAT(-1)) spaces. Combinatorially, one looks for local combinatorial conditions implying
some global features typical for nonpositively curved metric spaces.

A very important combinatorial condition of this type was formulated by Gromov \cite{Gro} for cubical complexes, i.e. cellular complexes
with cells being cubes. Namely, simply connected cubical complexes with links (that can be thought as small spheres around vertices)
being flag (respectively, $5$-large, i.e.\ flag-no-square) simplicial complexes carry a canonical CAT(0) (respectively, CAT(-1)) metric.
Another important local combinatorial condition is local $k$-largeness, introduced by Januszkiewicz-{\' S}wi{\c a}tkowski \cite{JS1} and Haglund \cite{Hag}. A flag simplicial complex is \emph{locally $k$-large} if its links do not contain `essential' loops of length less than $k$.
In particular, simply connected locally $7$-large simplicial complexes, i.e.\ \emph{$7$-systolic} complexes, are Gromov hyperbolic \cite{JS1}.
The theory of \emph{$7$-systolic groups}, that is groups acting geometrically on $7$-systolic complexes, allowed to provide important examples of highly dimensional Gromov hyperbolic groups \cite{JS0,JS1,O-chcg,OS}.

However, for groups acting geometrically on CAT(-1) cubical complexes or on $7$-systolic complexes, some very restrictive limitations
are known. For example, $7$-systolic groups are in a sense `asymptotically hereditarily aspherical', i.e.\ asymptotically they can not contain
essential spheres. This yields in particular that such groups are not fundamental groups of negatively curved manifolds of dimension
above two; see e.g. \cite{JS2,O-ci,O-ib,OS,Gom,O-ns}.
This rises need for other combinatorial conditions, not imposing restrictions as above. In \cite{O-sdn,ChOs,BCCGO,ChaCHO} some conditions
of this type are studied -- they form a way of unifying CAT(0) cubical and systolic theories.
On the other hand, Osajda \cite{O-8loc} introduced a local combinatorial condition of \emph{$8$-location}, and used it to provide a new solution to Thurston's
problem about hyperbolicity of some $3$-manifolds.

In the current paper we undertake a systematic study of a version of $8$-location, suggested in \cite[Subsection 5.1]{O-8loc}.
This version is in a sense more natural than the original one (tailored to Thurston's problem), and neither of them is implied by the other.
However, in the new $8$-location we do allow essential $4$-loops. This suggests that it can be used in a much wider context.
 Roughly (see Section $2$ for the precise definition), the new $8$--location says that essential loops of length at most $8$ admit filling diagrams with one internal vertex.

We show that this local combinatorial condition is a negative-curvature-type condition, by proving the following main result of the paper.

\begin{theorem}
Let $X$ be a simply connected, $8$-located simplicial complex. Then the $1$-skeleton of $X$, equipped
with the standard path metric, is Gromov hyperbolic.
\end{theorem}

The above theorem was announced without a proof in \cite[Subsection 5.1]{O-8loc}. In \cite[Subsection 5.2]{O-8loc} applications concerning
some weakly systolic complexes and groups are mentioned.

Our proof consists of two steps. In Theorem \ref{3.7} we show that an $8$-located simplicial complex satisfying a global condition $SD'$ (see Definition \ref{def-2.2}) is Gromov hyperbolic.
Then, in Theorem \ref{4.1} we show that the universal cover of an $8$-located complex satisfies the $SD'$ property. The main Theorem \ref{4.2} follows immediately from those two results.
For proving Theorem \ref{4.1} we use a method of constructing the universal cover introduced in \cite{O-sdn}, and then developed in \cite{BCCGO,ChaCHO}.

It is interesting to note that $8$-location is not a property preserved by taking covers. So it is possible to build an $8$-located complex which has an universal cover that is not $8$-located and hence it is not necessarily Gromov hyperbolic. This is in contrast with other curvature conditions mentioned above (CAT(0) metric, local $k$-largeness, the SD$_{2}^{\ast}$ property, bucolic complexes, weakly modular graphs). For such curvature conditions one can find results of the form: if a complex has some local curvature condition then its universal cover satisfies the same condition globally.

Unfortunately, at the moment we do not have examples to justify why $8$-location provides new examples of hyperbolic groups of higher dimension (i.e. not 'asymptotically hereditarily aspherical').

\section{Preliminaries}

Let $X$ be a simplicial complex. We denote by $X^{(k)}$ the $k$-skeleton of
$X, 0 \leq k < \dim X$. A subcomplex $L$ in $X$ is called \emph{full} as a subcomplex of $X$ if any simplex of $X$ spanned by a set of vertices in $L$, is a simplex of $L$. For a set
$A = \{ v_{1}, ..., v_{k} \}$ of vertices of $X$, by $\langle  A \rangle$ or by $\langle  v_{1}, ..., v_{k} \rangle$ we denote the \emph{span} of $A$, i.e. the
smallest full subcomplex of $X$ that
contains $A$. We write $v \sim v'$ if $\langle  v,v' \rangle \in X$ (it can happen that $v = v'$). We write $v \nsim v'$ if $\langle  v,v' \rangle \notin X$.
 We call $X$ {\it flag} if any finite set of vertices which are pairwise connected by
edges of $X$, spans a simplex of $X$.

A {\it cycle} ({\it loop}) $\gamma$ in $X$ is a subcomplex of $X$ isomorphic to a triangulation of $S^{1}$. A \emph{full cycle} in $X$ is a cycle that is full as a subcomplex of $X$.
A $k$-\emph{wheel} in $X$ $(v_{0}; v_{1}, ..., v_{k})$ (where $v_{i}, i \in \{0,..., k\}$
are vertices of $X$) is a subcomplex of $X$ such that $\gamma = (v_{1}, ..., v_{k})$ is a full cycle and $v_{0} \sim v_{1}, ..., v_{k}$.
The \emph{length} of $\gamma$ (denoted by $|\gamma|$) is the number of edges in $\gamma$. If $g = (v_{1}, ..., v_{k})$ is a $1$-skeleton geodesic of $X$, the \emph{length} of $g$ (denoted by $|g|$ or by $|(v_{1}, ..., v_{k})|$) is the number of edges in $g$.

We define the \emph{combinatorial metric} on the $0$-skeleton of $X$ as the number of edges in the shortest $1$-skeleton path joining two given vertices.
There are two distance functions we consider in this paper. The first is the combinatorial
distance between vertices in a simplicial complex, and the second is the length of geodesics
between points in a metric space. We will specify
which distance function we are using by writing $d_{c}(\cdot, \cdot)$, and $d(\cdot, \cdot)$, respectively.

A \emph{ball (sphere)}
$B_{i}(v,X)$ ($S_{i}(v,X)$) of radius $i$ around some vertex $v$ is a full subcomplex of $X$ spanned by vertices at combinatorial distance at most $i$ (at combinatorial distance $i$) from $v$.

\begin{definition}\label{def-2.1}
Let $O$ and $O'$ be two vertices of $X$ such that $d_{c}(O,O')= n$. Let $I$ be the set of vertices lying on geodesics between $O$ and $O'$. Let $I_{k} = S_{k}(O) \cap S_{n-k}(O') = S_{k}(O) \cap I$ for all $k \leq n$. We call $I$ an \emph{interval}. If for every $k \leq n$ and for every two vertices $v,w \in I_{k},$ we have $d_{c}(v,w) \leq j$, then we say $I$ is $j$\emph{-thin}.
\end{definition}

\begin{definition}\label{def-2.2}
A simplicial complex is $m$-\emph{located} if it is flag and every full homotopically trivial loop of length at most $m$ is contained in a $1$-ball.
\end{definition}

Let $\sigma$ be a simplex of $X$. The \emph{link} of $X$ at $\sigma$, denoted $X_{\sigma}$, is the subcomplex of $X$ consisting of all simplices of $X$ which are disjoint from $\sigma$ and which, together
with $\sigma$, span a simplex of $X$. We call a flag simplicial complex $k$\emph{-large} if there are no
full $j$-cycles in $X$, for $j<k$. We say $X$ is \emph{locally} $k$\emph{-large} if all its links are $k$-large.

We introduce further a global combinatorial condition on a flag simplicial complex.

\begin{definition}\label{def-2.3}
Let $X$ be a flag simplicial complex. For a vertex $O$ of $X$ and a natural number $n$, we say that $X$ satisfies \emph{the property $SD'_{n}(O)$}
if for every $i \in \{1, ..., n\}$ we have:
\begin{enumerate}
\item (T) (triangle condition): for every edge $e \in S_{i+1}(O)$, the intersection $X_{e} \cap B_{i}(O)$ is non-empty;
\item (V) (vertex condition): for every vertex $v \in S_{i+1}(O)$, and for every two vertices $u,w \in X_{v} \cap B_{i}(O)$,
there exists a vertex $t \in X_{v} \cap B_{i}(O)$ such that $t \sim u,w$. \end{enumerate}
We say $X$ satisfies \emph{the property $SD'(O)$} if $SD'_{n}(O)$ holds for each natural number $n$. We say $X$ satisfies \emph{the property $SD'$} if
$SD'_{n}(O)$ holds for each natural number $n$ and for each vertex $O$ of $X$.
\end{definition}

The following result is given in \cite{O-8loc}.

\begin{prop}\label{2.1}
A simplicial complex which satisfies the property $SD'(O)$ for some vertex $O$, is simply connected.
\end{prop}

By a \emph{covering} we mean a \emph{simplicial covering}, i.e.$\ $ a simplicial map restricting to isomorphisms from $1$-balls onto their images.

The following procedure will be applied frequently when proving the paper's main results.

\begin{definition}\label{def-2.4}
Given a path $\gamma = (v_{0}, v_{1}, ..., v_{n})$ in a simplicial complex $X$, one can \emph{tighten} it to a full path $\gamma'$ with the same endpoints by repeatedly applying the following operations:

$\bullet$ if $v_{i}$ and $v_{j}$ are adjacent in $X$ for some $j > i+1$, then remove from the sequence all $v_{k}$ where $i < k < j$;

$\bullet$ if $v_{i}$ and $v_{j}$ coincide in $X$ for some $j > i$, then remove from the sequence all $v_{k}$ where $i < k \leq j$.

We call $\gamma'$ a \emph{tightening} of $\gamma$. We allow the trivial case when $\gamma$ is already full. Then its tightening is the path itself.
\end{definition}

A loop $\gamma$ can be specified by taking the union of a pair of full paths $\gamma'$ and $\gamma''$ which have the same endpoints. The resultant loop is full if and only if no internal vertex (i.e. non-endpoint) of $\gamma'$ coincides with or is adjacent to any internal vertex of $\gamma''$.

We define next bigons between two points of a simplicial complex. First we consider these points to lie on edges of the complex. Secondly we consider the points to be vertices of the complex. Thirdly we consider the points to be one a vertex of the complex and the second the midpoint of an edge of the complex. In the first two cases the bigons have even length. In the third case the bigon has odd length.

Let $x$ and $x'$ be points of $X$ such that $x$ lies inside an edge $\langle v_{1}, w_{1} \rangle$, while $x'$ lies inside another edge $\langle v_{n}, w_{n} \rangle$. The points $x$ and $x'$ are chosen such that there are two distinct $1$-skeleton geodesics $\gamma_{1}$ and $\gamma_{2}$ of equal length joining them.
Let $\langle x,w_{1} \rangle = \{ y \in \langle v_{1},w_{1} \rangle | d(x,w_{1}) = d(x,y) + d(y,w_{1}) \}$. Note that $\langle x,w_{1} \rangle \subset \langle v_{1},w_{1} \rangle$.
Let $\gamma_{1} = \langle x,v_{1} \rangle \cup (v_{1}, ..., v_{n}) \cup \langle v_{n},x' \rangle$ and let $\gamma_{2} = \langle x,w_{1} \rangle \cup (w_{1}, ..., w_{n}) \cup \langle w_{n},x' \rangle$. We denote by $|\gamma_{1}| = d(x,v_{1}) + |(v_{1}, ..., v_{n})| + d(v_{n},x')$ the \emph{length} of $\gamma_{1}$. Let $B$ be the union of $\gamma_{1}$ and $\gamma_{2}$. We call $B$ a \emph{bigon} between the points $x$ and $x'$ in the metric on $X^{(1)}$ in which each edge has length $1$. The \emph{length} of the bigon $B$ is equal to $|\gamma_{1}| + |\gamma_{2}| = 2 \cdot |\gamma_{1}|$. We say $B$ is \emph{of even length}.

 Let $w_{1}$ and $v_{n}$ be vertices of $X$ such that there are two distinct $1$-skeleton geodesics $\gamma_{1}$ and $\gamma_{2}$ of equal length joining them. Let $\gamma_{1} = \langle w_{1},v_{1} \rangle \cup (v_{1}, ..., v_{n})$ and let $\gamma_{2} = (w_{1}, ..., w_{n}) \cup \langle w_{n},v_{n} \rangle$. Let $B$ be the union of $\gamma_{1}$ and $\gamma_{2}$. We call $B$ a \emph{bigon} between the vertices $w_{1}$ and $v_{n}$ in the metric on $X^{(1)}$ in which each edge has length $1$. The \emph{length} of the bigon $B$ is an even integer which is equal to $|\gamma_{1}| + |\gamma_{2}| = 2 \cdot |\gamma_{1}|$. We say $B$ is \emph{of even length}.
 We say that the geodesics $\gamma_{1}$ and $\gamma_{2}$ are $k$\emph{-close} and that the bigon $B$ is $k$\emph{-thin} if $d_{c}(v_{i},w_{i+1}) \leq k, \forall i \in \{ 1, ..., n-1 \}$.

 Let $v$ be a vertex of $X$ and let $m$ be the midpoint of en edge $\langle v_{n},w_{n} \rangle$ of $X$. The points $v$ and $m$ are chosen such that there are two distinct $1$-skeleton geodesics $\gamma_{1} = (v,v_{1}, ..., v_{n}) \cup \langle v_{n},m \rangle$ and $\gamma_{2} = (v, w_{1}, ..., w_{n}) \cup \langle w_{n},m \rangle$ of equal length between them. Let $B$ be the union of $\gamma_{1}$ and $\gamma_{2}$. We call $B$ a \emph{bigon} between the points $v$ and $m$ in the metric on $X^{(1)}$ in which each edge has length $1$. The \emph{length} of the bigon $B$ is an odd integer which is equal to $|\gamma_{1}| + |\gamma_{2}| = 2 \cdot |(v, v_{1}, ..., v_{n})| + 1$. We say $B$ is \emph{of odd length}. We say that the geodesics $\gamma_{1}$ and $\gamma_{2}$ are $k$\emph{-close} and that the bigon $B$ is $k$\emph{-thin} if $d_{c}(v_{i},w_{i}) \leq k, \forall i \in \{ 1, ..., n \}$.

\section{Hyperbolicity}

In this section we show that the $8$-location on a simplicial complex enjoying the $SD'$ property, implies Gromov hyperbolicity. We start with two useful lemmas.

\begin{lemma}\label{3.1}
Let $X$ be an $8$-located simplicial complex which satisfies the $SD'_{n}(O)$ property for some vertex $O$, $n \geq 2$. Let $v \in S_{n+1}(O)$ and let $y,z \in B_{n}(O)$ be such that $v \sim y,z$
and $d_{c}(y,z) = 2$.
Let $w \in B_{n}(O)$ be a vertex such that $w \sim y, v, z$, given by the vertex condition (V). Consider the vertices $u_{1}, u_{2} \in B_{n-1}(O)$ such that
$u_{1} \sim  y, w$ and $u_{2} \sim  w, z$, given by the triangle condition (T). If $u_{1} \nsim z$ and $u_{2} \nsim y$, then
$u_{1} \sim u_{2}$.
\end{lemma}

\begin{proof}
The proof is by contradiction. Assume that $d_{c}(u_{1},u_{2}) = 2$. Let $u' \in B_{n-1}(O)$ be a vertex such that $u' \sim u_{1}, w, u_{2}$, $u' \neq u_{1}$, $u' \neq u_{2}$, given by the vertex condition (V).
Let $t_{1}, t_{2} \in B_{n-2}(O)$ be vertices such that
$\langle  u_{1}, t_{1}, u' \rangle,  \langle  u', t_{2} ,u_{2} \rangle \in X$, given by the triangle condition (T). Let $t' \in B_{n-2}(O)$ be a vertex such that
$t' \sim t_{1}, u', t_{2}$ (possibly with $t' = t_{2}$),
given by the vertex condition (V).

Note that $v \in S_{n+1}(O), y,z \in B_{n}(O), u_{1}, u', u_{2} \in B_{n-1}(O)$, and $t_{1}, t', t_{2} \in B_{n-2}(O)$.
Hence there are no edges between the vertices $y,v,z$ and the vertices $t_{1}, t', t_{2}$. Also there are no edges between the vertices $u_{1}, u', u_{2}$ and the vertex $v$.
Let $\gamma'$ be the path $(u_{1}, y, v, z, u_{2})$ which is full because $u_{1} \nsim z$ and $u_{2} \nsim y$. Let $\gamma''$ be a tightening of the path $(u_{2}, t_{2}, t', t_{1}, u_{1})$.
Note that $\gamma = \gamma' \cup \gamma''$ is a full loop of length at most $8$.

By $8$-location, it follows
that $\gamma$ is contained in the link of a vertex. This implies that $d_{c}(v, t_{1}) = 2$, respectively that $d_{c}(v, t') = 2$. But since $v \in S_{n+1}(O)$ and $t_{1} \in B_{n-2}(O)$, respectively $t' \in B_{n-2}(O)$, we have that $d_{c}(v, t_{1}) = 3$,
respectively that $d_{c}(v, t') = 3$.
This yields a contradiction. So $u_{1} \sim u_{2}$.

\begin{figure}[h]
    \begin{center}
       \includegraphics[height=2.5cm]{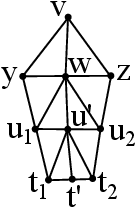}
      \caption{}
    \end{center}
\end{figure}

\end{proof}

\begin{lemma}\label{3.2}
Let $X$ be an $8$-located simplicial complex which satisfies the $SD'_{n}(O)$ property for some vertex $O$, $n \geq 2$. Let $v_{1}, v_{2}, v_{3} \in B_{n-1}(O)$ be such that $v_{1} \sim v_{2} \sim v_{3}$. Let $w_{1},w_{2} \in B_{n-2}(O)$ be such that $w_{1} \sim v_{1},v_{2}$ and
$w_{2} \sim v_{2},v_{3}$, given by the triangle condition (T).
Let $p_{1},p_{2} \in B_{n}(O)$ be such that $p_{1} \sim v_{1},v_{2}$ and
$p_{2} \sim v_{2},v_{3}$, given by the triangle condition (T).
Then $w_{1} \sim w_{2}$ if and only if
$p_{1} \sim p_{2}$.
\end{lemma}

\begin{proof}
We show that if $w_{1} \sim w_{2}$ then $p_{1} \sim p_{2}$. The inverse implication can be proven similarly.
Assume by contradiction that $d_{c}(p_{1},p_{2}) = 2$.
Let $p' \in B_{n}(O)$ be such that $p' \sim p_{1}, v_{2}, p_{2}$, $p' \neq p_{1}$, $p' \neq p_{2}$, given by the vertex condition (V). We consider the vertices $u_{1}, u_{2} \in B_{n+1}(O)$ such that
$u_{1} \sim  p_{1}, p'$ and $u_{2} \sim  p_{2}, p'$, given by the triangle condition (T). Because $d_{c}(p_{1},p_{2}) = 2$, Lemma \ref{3.1} implies that $u_{1} \sim u_{2}$. Let $\gamma'$ be a tightening of the path $(v_{1},p_{1},u_{1},u_{2},p_{2},v_{3})$ and let $\gamma''$ be a tightening of the path $(v_{3},w_{2},w_{1},v_{1})$. Then $\gamma = \gamma' \cup \gamma''$ is a full loop of length is at most $8$. By $8$-location, $\gamma$ is contained in the link of a vertex. So $d_{c}(w_{1},u_{1}) = 2$. But, since $w_{1} \in B_{n-2}(O)$ and $u_{1} \in B_{n+1}(O)$, we have $d_{c}(w_{1},u_{1}) = 3$. Because we have reached a contradiction, it follows  that $p_{1} \sim p_{2}$.

\begin{figure}[h]
    \begin{center}
       \includegraphics[height=2.5cm]{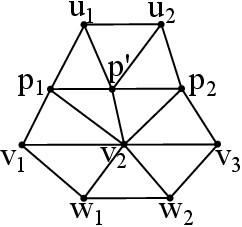}
      \caption{}
    \end{center}
\end{figure}

\end{proof}

In the next theorem we show the thinness of discrete intervals in an $8$-located simplicial complex satisfying the SD' property.

\begin{theorem}\label{3.3}
Let $X$ be an $8$-located simplicial complex which satisfies the $SD'$ property. Then intervals of $X$ are $2$-thin.
\end{theorem}

\begin{proof} Let $O$ and $O'$ be two vertices of $X$ such that $d_{c}(O,O')= n$. Let $I$ be an interval i.e. the set of vertices lying on geodesics between $O$ and $O'$.  For all $k \leq n$, let $I_{k} = S_{k}(O) \cap S_{n-k}(O') = S_{k}(O) \cap I$.
We prove by contradiction that for every $k \leq n$ and for every two vertices $v,w \in I_{k},$ we have $d_{c}(v,w) \leq 2$.

We build a full path in $I_{k}$ of diameter $3$ as in \cite{O-8loc}. Suppose there are vertices $v,w \in I_{k}$ such that $d_{c}(v,w) > 2$. Let $k$ be the maximal natural number for which this happens. Then there exist vertices $v',w'$ in $I_{k+1}$ such that
$v' \sim v$, $w' \sim w$, and $d_{c}(v',w') \leq 2$.

By the vertex condition (V), there is a vertex $z$ in $I_{k+1}$ such that $z \sim v',w'$, possibly with $z = w'$. By the triangle condition (T), there are vertices
$v'',w'' \in I_{k}$ such that $\langle  v', v'', z \rangle, \langle  z, w'', w' \rangle \in X$ (with $v'' = w''$ if $z=w'$).
By the vertex condition (V), there are vertices $s,t$ and $u$ in $I_{k}$ such that $s \sim v,v',v''; t \sim v'',z,w''; u \sim w,w',w''$ (possibly with $s=v'', t = w''$, and $u=w$).
Among the vertices $t,w'',u,w$ we choose the first one (in the given order), that is at distance $3$ from $v$.
Denote this vertex by $v'''$. In this way we obtain a full path $(v_{1}, v_{2}, v_{3}, v_{4})$ in $I_{k}$ of diameter $3$, with $v_{1} = v$ and
$v_{4} = v'''$. We will show that such a path can not exist reaching hereby a contradiction and proving the theorem.

Note that because $X$ satisfies the $SD'$ property, it is, by Proposition \ref{2.1}, simply connected. Any loop in $X$ is therefore homotopically trivial.

By the triangle condition (T), there exist vertices $w_{i}$ in $I_{k-1}$ such that $\langle  v_{i}, w_{i},$ $v_{i+1} \rangle \in X, 1 \leq i \leq 3$. We discuss further all possible cases
(and the corresponding subcases) of mutual relations between the vertices $w_{i}, 1 \leq i \leq 3$ (up to renaming vertices).
Case I is when $w_{1} = w_{2}$.
Case II is when $w_{1} \neq w_{2} \neq v_{3}$.

By the triangle condition (T), there are vertices $p_{i}$ in $I_{k+1}$ such that $\langle  v_{i}, p_{i}, v_{i+1} \rangle$ $\in X, 1 \leq i \leq 3$
(possibly with $p_{1} = p_{2}$ or $p_{2} = p_{3}$).

Note that if $p_{1} = p_{3}$ or $w_{1} = w_{3}$, then $d_{c}(v_{1},v_{4}) = 2$. This implies contradiction with $d_{c}(v_{1},v_{4}) = 3$. If $v_{1} \sim w_{3}$ or $v_{4} \sim w_{1}$ we also get contradiction with $d_{c}(v_{1},v_{4}) = 3$. If $v_{1} \sim w_{2}$ and $v_{4} \sim w_{2}$, contradiction with $d_{c}(v_{1},v_{4}) = 3$ yields as well.

\subsection{Case I}

We start treating the case when $w_{1} = w_{2}$ which has two subcases. Case I.$1$ is when $w_{2} \sim w_{3}$. Case I.$2$ is when $d_{c}(w_{2}, w_{3}) = 2$.
Since $d_{c}(v_{1}, v_{3}) = 2$, Lemma \ref{3.1} implies that $p_{1} \sim p_{2}$.

\subsubsection{Case I.$1$}

We treat the case when $w_{1} = w_{2}, w_{2} \sim w_{3}$. Lemma \ref{3.2} implies that $p_{2} \sim p_{3}$.

Let $\gamma'$ be the path $(v_{1}, w_{2}, w_{3}, v_{4})$ which is full. Let $\gamma''$ be a tightening of the path $(v_{4}, p_{3},$ $ p_{2}, p_{1}, v_{1})$. Note that $w_{2}, w_{3} \in I_{k-1}$ and $p_{1}, p_{2}, p_{3} \in I_{k+1}$. Hence there are no edges between the vertices $w_{2}, w_{3}$ and the vertices $p_{1}, p_{2}, p_{3}$. Then $\gamma = \gamma' \cup \gamma''$ is a full loop of length at most $7$. By $8$-location it follows that $\gamma$ is contained in the link of a vertex. So $d_{c}(v_{1}, v_{4}) = 2$ which yields contradiction with $d_{c}(v_{1}, v_{4}) = 3$.

\begin{figure}[h]
    \begin{center}
       \includegraphics[height=2cm]{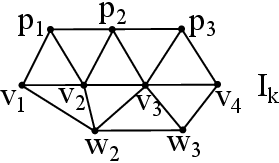}
      \caption{Case I.$1$}
    \end{center}
\end{figure}

\subsubsection{Case I.$2$}

We treat further the case when $w_{1} = w_{2}, d_{c}(w_{2}, w_{3}) = 2$.
By the vertex condition (V), there exists a vertex $w' \in I_{k-1}$ such that $w' \sim w_{2}, v_{3}, w_{3}$.

Note that if $v_{1} \sim w'$ and $v_{4} \sim w'$, we have contradiction with $d_{c}(v_{1}, v_{4}) = 3$.
The situation $v_{1} \sim w'$ but $v_{4} \nsim w'$ can be treated similar to case I.$1$. The situation $v_{4} \sim w'$
but $v_{1} \nsim w'$ can be also treated similar to case I.$1$. Assume that from now on such situations do not occur.

Assume that $v_{2} \sim w_{3}$. Because $d_{c}(v_{2}, v_{4}) = 2$, Lemma \ref{3.1} implies that $p_{2} \sim p_{3}$.
We consider the path $\beta' = (v_{1}, w_{2}, w', w_{3}, v_{4})$ which is full. Let $\beta''$ be a tightening of the path $(v_{4}, p_{3}, p_{2}, p_{1}, v_{1})$. Since the loop $\beta' \cup \beta''$ is full and its length is at most $8$, by $8$-location, we get contradiction with $d_{c}(v_{1}, v_{4}) = 3$. From now on assume that $v_{2} \nsim w_{3}$.

\begin{figure}[ht]
    \begin{center}
        \includegraphics[height=2.5cm]{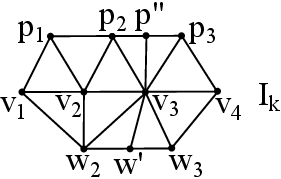}
      \caption{Case I.$2$}
    \end{center}
\end{figure}

We consider the path $\gamma' = (v_{1}, w_{2}, w', w_{3}, v_{4})$ which is full. Let $\gamma''$ be a tightening of the path $(v_{4}, p_{3}, p'',$ $ p_{2}, p_{1}, v_{1})$. Note that the loop $\gamma = \gamma' \cup \gamma''$ is full and its length is at most $9$.
If the length of $\gamma$ is at most $8$, it is contained, by $8$-location, in the link of a vertex. Then we get contradiction with $d_{c}(v_{1}, v_{4}) = 3$.
For the rest of case I.$2$ we treat the case when $|\gamma| = 9$.  Note that then $d_{c}(p_{2}, p_{3}) = 2$ and $v_{1} \nsim p''$. By the triangle condition (T), there are vertices $q_{1}, q_{2}$ in $I_{k+2}$ such that $\langle  p_{1}, q_{1}, p_{2} \rangle,$ $ \langle  p_{2}, q_{2}, p'' \rangle $ $\in X$. Because $v_{2} \sim v_{3}$, Lemma \ref{3.2} implies that $q_{1} \sim q_{2}$.

$\bullet$ Assume that $p_{1} \nsim v_{3}$.
We consider the path $\beta' = (p_{1}, v_{1}, w_{2}, v_{3}, p'')$. Let $\beta''$ be a tightening of the path $(p'', q_{2}, q_{1}, p_{1})$.
Note that $\beta = \beta' \cup \beta''$ is full. Then because $\beta$ has length at most $7$,
it follows, by $8$-location, that it is contained in the link of a vertex. So $d_{c}(q_{1}, w_{2}) = 2$. Because $d_{c}(q_{1}, w_{2}) = 3$, we have reached a contradiction.

\begin{figure}[h]
    \begin{center}
        \includegraphics[height=2.8cm]{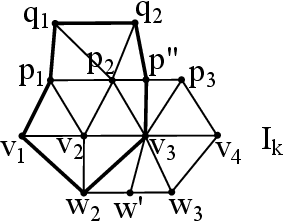}
      \caption{Case I.$2$: $|\gamma| = 9$, $d_{c}(p_{2}, p_{3}) = 2, v_{1} \nsim p'', p_{1} \nsim v_{3}$}
    \end{center}
\end{figure}

$\bullet$ Assume that $p_{1} \sim v_{3}$. By the triangle condition $(T)$,
there is a vertex $q_{3} \in I_{k +2}$ such that $\langle p'', q_{3}, p_{3} \rangle \in X$.
Because $d_{c}(p_{2},p_{3}) = 2$, Lemma \ref{3.1} implies that $q_{2} \sim q_{3}$.
We consider the full path $\beta' = (p_{1}, v_{3}, p_{3})$. Let $\beta''$ be a tightening of the path $(p_{3}, q_{3}, q_{2}, q_{1}, p_{1})$. Because $\beta = \beta' \cup \beta''$ is full and it has length at most $6$, by $8$-location, there exists a vertex $p_{02}$ adjacent to all vertices
of $\beta$. Note that $p_{02} \in I_{k+1}$. We consider the full path $\alpha' = (v_{1}, w_{2}, w', w_{3}, v_{4})$. Let
$\alpha''$ be a tightening of the path $(v_{4}, p_{3}, p_{02}, p_{1}, v_{1})$. Because $\alpha' \cup \alpha''$ is full and it has length at most $8$, by $8$-location, we get contradiction with $d_{c}(v_{1}, v_{4}) = 3$.

\begin{figure}[ht]
    \begin{center}
        \includegraphics[height=2.8cm]{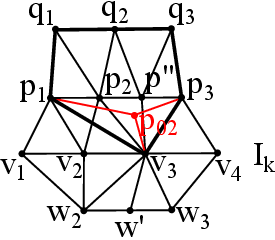}
      \caption{Case I.$2$: $|\gamma| = 9$, $d_{c}(p_{2}, p_{3}) = 2, v_{1} \nsim p'', p_{1} \sim v_{3}$}
    \end{center}
\end{figure}

\subsection{Case II}

We treat next the case when $w_{1} \neq w_{2} \neq w_{3}$ which has $3$ subcases. Case II.$1$ is when $w_{1} \sim w_{2} \sim w_{3}$. Case II.$2$ is when
$w_{1} \sim w_{2}$ and $d_{c}(w_{2}, w_{3}) = 2$.
Case II.$3$ is when $d_{c}(w_{1}, w_{2}) = d_{c}(w_{2}, w_{3}) = 2$.

The situation $v_{3} \sim w_{1}$ can be treated like case I. The case $v_{1} \sim w_{2}$ but $v_{4} \nsim w_{2}$ can be treated like case I$.1$. From now on assume that such situations do not occur.

We consider the case when $w_{1} \sim w_{3}$. We consider the path $\gamma' = (v_{1}, w_{1}, w_{3}, v_{3})$ which is full. Let $\gamma''$ be a tightening of the path $(v_{3}, p_{2}, p', p_{1}, v_{1})$. Since the loop $\gamma = \gamma' \cup \gamma''$ is full and its length is at most $7$, by $8$-location,
 $\gamma$ is contained in the link of a vertex
$v_{02}$. Note that $v_{02} \in I_{k}$. If $v_{02} \sim v_{4}$ we have contradiction with $d_{c}(v_{1}, v_{4}) = 3$. Because the path $(v_{1}, v_{2}, v_{3}, v_{4})$ is full,
and $v_{02} \sim v_{1}, v_{3}$, the path $(v_{1}, v_{02}, v_{3}, v_{4})$ is also full.
Moreover, since $\langle  v_{1}, w_{1}, v_{02} \rangle$,
$\langle  v_{02}, w_{3}, v_{3} \rangle$, $\langle  v_{3}, w_{3}, v_{4}\rangle \in X$ and $w_{1} \sim w_{3}$, we are in case I$.1$.
From now on assume that $w_{1} \nsim w_{3}$.

\begin{figure}[ht]
    \begin{center}
        \includegraphics[height=2.5cm]{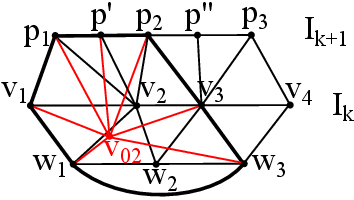}
      \caption{Case II: $w_{1} \sim w_{3}$ }
    \end{center}
\end{figure}

\subsubsection{Case II.$1$}

We treat the case when $w_{1} \sim w_{2} \sim w_{3}$. Lemma \ref{3.2} implies that $p_{1} \sim p_{2} \sim p_{3}$.

\begin{figure}[h]
    \begin{center}
        \includegraphics[height=2.5cm]{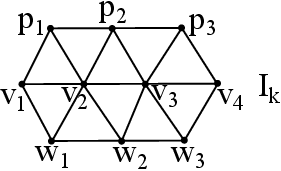}
      \caption{Case II.$1$}
    \end{center}
\end{figure}

Let $\gamma' = (v_{1},w_{1},w_{2},w_{3},v_{4})$ and let $\gamma''$ be a tightening of the path $(v_{4},p_{3},p_{2},p_{1},v_{1})$. Note that the loop $\gamma = \gamma' \cup \gamma''$ is full. Thus, since its length is at most $8$, $\gamma$ is contained, by $8$-location, in the link of a vertex. This yields contradiction with $d_{c}(v_{1},v_{4}) = 3$.

\subsubsection{Case II.$2$}

Suppose $w_{1} \sim w_{2}$ and $d_{c}(w_{2}, w_{3}) = 2$. By the vertex condition (V), there is a vertex $w'$ in $I_{k-1}$ such that $w' \sim w_{2}, v_{3}, w_{3}$ (possibly with $w' = w_{3}$).

Note that if $v_{1} \sim w'$ and $v_{4} \sim w'$, we have contradiction with $d_{c}(v_{1}, v_{4}) = 3$. The case $v_{1} \sim w'$ but $v_{4} \nsim w'$ can be treated like case I.$1$. The case $v_{4} \sim w'$ but $v_{1} \nsim w'$ can be treated like case II.$1$. Assume that from now on such situations do not occur.

By the triangle condition (T), there are vertices $u_{1}, u_{2}, u_{3} \in I_{k-2}$ such that
$\langle  w_{1},$ $ u_{1}, w_{2} \rangle,$ $\langle w_{2},  u_{2}, w' \rangle, \langle  w', u_{3}, w_{3} \rangle \in X$. Because $v_{2} \sim v_{3}$, Lemma \ref{3.2} implies that $u_{1} \sim u_{2}$. Because $d_{c}(w_{2},w_{3}) = 2$, Lemma \ref{3.1} implies that $u_{2} \sim u_{3}$.

Let $\gamma'$ be the path $(w_{1}, v_{2}, v_{3}, w_{3})$ which is full. Let $\gamma''$ be a tightening of the path $(w_{3}, u_{3}, u_{2}, u_{1}, w_{1})$. Because $\gamma = \gamma' \cup \gamma''$ is full and it has length at most $7$, by $8$-location, there is a vertex $w_{02}$ adjacent to all vertices of $\gamma$. Note that $w_{02} \in I_{k-1}$.
Since $\langle  v_{1}, w_{1}, v_{2} \rangle, \langle  v_{2}, w_{02}, v_{3} \rangle,$
$\langle  v_{3}, w_{3}, v_{4} \rangle$ $\in X$ and $w_{1} \sim w_{02} \sim w_{3}$, we are in case II$.1$.

\begin{figure}[h]
    \begin{center}
        \includegraphics[height=2.5cm]{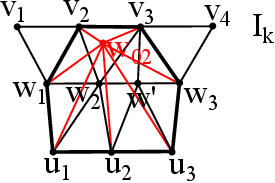}
      \caption{Case II.$2$}
    \end{center}
\end{figure}

\subsubsection{Case II.$3$}

Suppose $d_{c}(w_{1}, w_{2}) = d_{c}(w_{2}, w_{3}) = 2$.
By the vertex condition (V), there are vertices $w',w''$ in $I_{k-1}$ such that $w' \sim w_{1}, v_{2}, w_{2}$, and $w'' \sim w_{2}, v_{3}, w_{3}$.

Note that if $v_{1} \sim w'$ and $v_{4} \sim w'$, we have contradiction with $d_{c}(v_{1}, v_{4}) = 3$. The case $v_{1} \sim w'$ but $v_{4} \nsim w'$ can be treated like case II.$2$. The case $v_{4} \sim w'$ but $v_{1} \nsim w'$ can be treated like case I.$1$. The case $v_{1} \sim w_{2}$, $v_{4} \nsim w_{2}$ can be treated like case I. Assume that from now on such situations do not occur.

By the triangle condition (T), there are vertices $u_{1}, u_{2}, u_{3}, u_{4} \in I_{k-2}$ such that $\langle  w_{1}, u_{1}, w' \rangle, \langle  w', u_{2}, w_{2} \rangle,$
$\langle  w_{2}, u_{3}, w'' \rangle, \langle  w'', u_{4}, w_{3} \rangle \in X$. Because $d_{c}(w_{1},w_{2}) = d_{c}(w_{2},w_{3}) = 2$, Lemma \ref{3.1} implies that $u_{1} \sim u_{2}$ and $u_{3} \sim u_{4}$.
Because $v_{2} \sim v_{3}$, Lemma \ref{3.2} implies that $u_{2} \sim u_{3}$.

Let $\gamma' = (w_{1}, v_{2}, v_{3}, w_{3})$ which is full. Let $\gamma''$ be a tightening of the path $(w_{3}, u_{4}, u_{3}, u_{2}, u_{1}, w_{1})$.
Because $\gamma = \gamma' \cup \gamma''$ is full and it has length at most $8$, by $8$-location, there is a vertex $w_{02}$ adjacent to all vertices of $\gamma$. Note that $w_{02} \in I_{k-1}$.
Since $\langle  v_{1}, w_{1}, v_{2} \rangle, \langle  v_{2}, w_{02}, v_{3} \rangle,
\langle  v_{3}, w_{3}, v_{4} \rangle$ $\in X$ and $w_{1} \sim w_{02} \sim w_{3}$, we are in case II$.1$.

\begin{figure}[ht]
    \begin{center}
        \includegraphics[height=2.4cm]{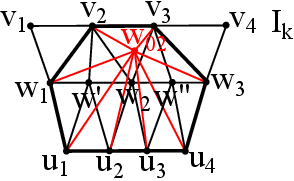}
      \caption{Case II.$3$}
    \end{center}
\end{figure}

\end{proof}

In the next lemma we show that a bigon of even length whose endpoints are points that belong to edges of a simplicial complex, can be considered as the union of two geodesics with endpoints at vertices of the complex.

\begin{lemma}\label{3.5}
Let $X$ be a simplicial complex. Let $x,x'$ be two points of $X$ such that $x$ lies inside one edge $\langle v_{1},w_{1} \rangle$, and $x'$ lies inside another edge $\langle v_{n},w_{n} \rangle$. Also $x$ and $x'$ are chosen such that there is a bigon of even length $B = \gamma_{1} \cup \gamma_{2}$ between them. Let $\gamma_{1} = \langle x,v_{1} \rangle \cup (v_{1}, ..., v_{n}) \cup \langle v_{n},x' \rangle$ and let $\gamma_{2} = \langle x,w_{1} \rangle \cup (w_{1}, ..., w_{n}) \cup \langle w_{n},x' \rangle$. Then $B = \beta_{1} \cup \beta_{2} = \gamma_{1} \cup \gamma_{2}$ (equality as sets), $\beta_{1}$ and $\beta_{2}$ being geodesics between the vertices $v_{1}$ and $w_{n}$ (or between the vertices $v_{n}$ and $w_{1}$).
\end{lemma}

\begin{proof} There are two cases.
The first case is when $d(v_{1},x) = d(v_{n},x') = \frac{1}{2}$. Then we may consider either $\beta_{1} =  (v_{1}, ..., v_{n}) \cup \langle v_{n},w_{n} \rangle$, $\beta_{2} = \langle v_{1},w_{1} \rangle \cup (w_{1}, ..., w_{n})$ or $\beta_{1} =  \langle w_{1},v_{1} \rangle \cup (v_{1}, ..., v_{n})$, $\beta_{2} = (w_{1}, ..., w_{n}) \cup \langle w_{n},v_{n} \rangle$.

The second case is when $d(v_{1},x) \neq \frac{1}{2}, d(v_{n},x') \neq \frac{1}{2}$. It has two subcases treated below.
We treat only the first subcase. The other subcase is similar.

\begin{enumerate}
\item The first subcase is when $d(v_{1},x) < d(x,w_{1})$. Since $B$ is a bigon, we have $d(x,v_{1}) + d(v_{n},x') = d(x,w_{1}) + d(w_{n},x')$. Hence \begin{equation}d(v_{n},x') > d(w_{n},x'). \end{equation} We choose the vertices $v_{1} \in \gamma_{1}$ and $w_{n} \in \gamma_{2}$ and we show that for these vertices the lemma holds.
\item The second subcase is when $d(v_{1},x) > d(x,w_{1})$. We choose the vertices $v_{n} \in \gamma_{1}$ and $w_{1} \in \gamma_{2}$.
\end{enumerate}

Suppose there is a geodesic $\gamma_{3}$ joining the vertices $v_{1}$ and $w_{n}$ which does not pass through $v_{n}$ such that $|\gamma_{3}| = |(v_{1}, ..., w_{n})| < |(v_{1}, ..., v_{n})| + d(v_{n},w_{n}) = |(v_{1}, ..., v_{n})| + 1$. Hence
\begin{equation}
|\gamma_{3}| \leq [|(v_{1}, ..., v_{n})| + 1] - 1 = |(v_{1}, ..., v_{n})|.
\end{equation}


 According to $(3.1)$ and $(3.2)$, we have:

 \begin{center}$d(x,v_{1}) + |(v_{1}, ..., w_{n})| + d(w_{n},x') = d(x,v_{1}) + |\gamma_{3}| + d(w_{n},x') < $
 $< d(x,v_{1}) + |\gamma_{3}| + d(v_{n},x') \leq d(x,v_{1}) + |(v_{1}, ..., v_{n})| + d(v_{n},x') = |\gamma_{1}|$.
 \end{center}
But $B = \gamma_{1} \cup \gamma_{2}$ is a bigon between the points $x$ and $x'$. Therefore $\gamma_{1}$ is a geodesic between $x$ and $x'$. The relation $d(x,v_{1}) + |(v_{1}, ..., w_{n})| + d(w_{n},x') < |\gamma_{1}|$ implies hence a contradiction. This ensures that there is no geodesic $\gamma_{3}$ between the vertices $v_{1}$ and $w_{n}$ as described above. So we have $B = \beta_{1} \cup \beta_{2} = \gamma_{1} \cup \gamma_{2}$ (equality as sets), $\beta_{1}$ and $\beta_{2}$ being geodesics between the vertices $v_{1}$ and $w_{n}$.


\end{proof}

We show further the thinness of bigons of even length in an $8$-located simplicial complex satisfying the SD' property.

\begin{lemma}\label{3.6}
Let $X$ be an $8$-located simplicial complex with the SD' property. Let $x,x'$ be either two vertices of $X$ or two points of $X$ such that $x$ lies inside one edge $\langle v_{1},w_{1} \rangle$ of $X$, while $x'$ lies inside another edge $\langle v_{n},w_{n} \rangle$ of $X$. The points $x$ and $x'$ are considered such that there is a bigon $B = \gamma_{1} \cup \gamma_{2}$ of even length between them. Then $B$ is $2$-thin.
\end{lemma}

\begin{proof} If $x$ and $x'$ are vertices of $X$, then the $2$-thinness of the bigon $B$ follows by Theorem \ref{3.3}.
If $x$ and $x'$ are not vertices of $X$, let $\gamma_{1} = \langle x,v_{1} \rangle \cup (v_{1}, ..., v_{n}) \cup \langle v_{n},x' \rangle$ and let $\gamma_{2} = \langle x,w_{1} \rangle \cup (w_{1}, ..., w_{n}) \cup \langle w_{n},x' \rangle$. According to the previous lemma, $B = \beta_{1} \cup \beta_{2} = \gamma_{1} \cup \gamma_{2}$ (equality as sets) where $\beta_{1}$ and $\beta_{2}$ are geodesics between the vertices $v_{1}$ and $w_{n}$ (or between the vertices $v_{n}$ and $w_{1}$). Then the lemma follows again by Theorem \ref{3.3}.

\end{proof}

In the next lemma we show the thinness of bigons of odd length in an $8$-located simplicial complex satisfying the SD' property.

\begin{lemma}\label{3.7} Let $X$ be an $8$-located simplicial complex with the SD' property. Let $B = \gamma_{1} \cup \gamma_{2}$ be a bigon of odd length between two points $x$ and $x'$ of $X$. Either both points $x$ and $x'$ lie inside edges of $X$ or one of them is a vertex of $X$ and the other is the midpoint of an edge of $X$. Then $B$ is $4$-thin.
\end{lemma}

\begin{proof} Let $X'$ be the barycentric subdivision of $X$. If $x$ and $x'$ lie inside edges of $X$ but are not vertices of $X'$, then one can show as in Lemma \ref{3.5} that $B = \beta_{1} \cup \beta_{2} = \gamma_{1} \cup \gamma_{2}$ (equality as sets), $\beta_{1}$ and $\beta_{2}$ being geodesics between two vertices of $X'$. Because $B$ has odd length, both endpoints of $\beta_{1}$ and $\beta_{2}$ can not be vertices of $X$. Let one endpoint of $\beta_{1}$ and $\beta_{2}$ be a vertex $O$ of $X$, and let the other endpoint be the midpoint $m$ of an edge $\langle x,y \rangle$ of $X$. Note that $O$ and $m$ are vertices of $X'$.
Since $B$ is a bigon between $O$ and $m$, we have $d_{c}(x,O) = d_{c}(y,O) = k+1$. Let $\alpha_{1} \subset \beta_{1}$ be a geodesic joining $O$ and $x$ of length $k+1$. Let $\alpha_{2} \subset \beta_{2}$ be a geodesic joining $O$ and $y$ of length $k+1$. For any natural number $n$, $X$ satisfies the $SD'_{n}(O)$ property. Then, by the triangle condition (T), for any $i \in \{1, ..., n\}$ and for any edge $e \subset S_{i+1}(O)$, we have $X_{e} \cap B_{i}(O) \neq \emptyset$. So for $\langle x,y \rangle \subset S_{k+1}(O)$, we have $X_{\langle x,y \rangle} \cap B_{k}(O) \neq \emptyset$. Hence there is a vertex $a \in X_{\langle x,y \rangle} \cap B_{k}(O)$. This implies that $d_{c}(a,O) = k$, $a \sim x$ and $a \sim y$.
So there is a geodesic $\alpha$ from $O$ to $a$ of length $k$. Therefore there is a geodesic $\alpha_{1}' = \alpha \cup \langle a,x \rangle$ from $O$ to $x$ of length $k+1$. Also there is a geodesic $\alpha_{2}' = \alpha \cup \langle a,y \rangle$ from $O$ to $y$ of length $k+1$. Let $B_{1} = \alpha_{1} \cup \alpha_{1}'$ be a bigon between $O$ and $x$. Let $B_{2} = \alpha_{2} \cup \alpha_{2}'$ be a bigon between $O$ and $y$. Both bigons $B_{1}$ and $B_{2}$ are of even length which is equal to $2 \cdot (k+1)$. Then the previous lemma implies that $\alpha_{1}$ is $2$-close to $\alpha_{1}'$, while $\alpha_{2}$ is $2$-close to $\alpha_{2}'$. Since $\alpha_{1}' = \alpha \cup \langle a,x \rangle$, $\alpha_{2}' = \alpha \cup \langle a,y \rangle$, while $x \sim y$, we note that $\alpha_{1}'$ is nearly the same as $\alpha_{2}'$. Thus $\alpha_{1}$ is $4$-close to $\alpha_{2}$. $B$ is therefore $4$-thin.

\end{proof}

We give the main result of the section.

\begin{theorem}\label{3.7}
Let $X$ be an $8$-located simplicial complex which satisfies the $SD'$ property. Then the $0$-skeleton of $X$ with a path metric induced from $X^{(1)}$, is $\delta$-hyperbolic, for a
universal constant $\delta$.
\end{theorem}

\begin{proof} The proof is similar to the one of the analogous Theorem $3.3$ given in \cite{O-8loc}.
According to \cite{Papa}, hyperbolicity of the $0$-skeleton follows by Lemma \ref{3.6} and by Lemma \ref{3.7}. These lemmas imply that in $X$ the bigons of even length are $2$-thin, while the bigons of odd length are $4$-thin. So all bigons of $X$ are $4$-thin. Hence we may conclude that the hyperbolicity constant is universal.

\end{proof}

\section{Local-to-global}

In this section we show that the universal cover of an $8$-located simplicial complex satisfies the $SD'$ property.

\begin{theorem}\label{4.1}
Let $X$ be an $8$-located simplicial complex. Then its universal cover $\widetilde{X}$ is a simplicial complex which satisfies the $SD'$ property.
\end{theorem}

\begin{proof} The proof is similar to the one of the analogous Theorem $3.4$ given in \cite{O-8loc}.

We construct the universal cover $\widetilde{X}$ of $X$ as an increasing union $\cup _{i=1}^{\infty}\widetilde{B}_{i}$ of combinatorial balls. The covering map is then the union
$\cup_{i=1}^{\infty} f_{i} : \cup_{i=1}^{\infty} \widetilde{B}_{i} \rightarrow X$
where $f_{i} : \widetilde{B}_{i} \rightarrow X$ is locally injective and $f_{i}| _{\widetilde{B}_{j}} = f_{j}$, for $j \leq i$.

The proof is by induction. We choose a vertex $O$ of $X$ and we define $\widetilde{B}_{0} = \{O\}, \widetilde{B}_{1} = B_{1}(O,X)$, and $f_{1} = Id _{B_{1}(O)}$. We assume that we have constructed
the balls $\widetilde{B}_{1}, \widetilde{B}_{2}, ..., \widetilde{B}_{i}$ and the corresponding maps $f_{1}, f_{2}, ..., f_{i}$ to $X$ such that the following conditions hold:
\begin{enumerate}
\item ($P_{i}$): $\widetilde{B}_{j} = B_{j}(O, \widetilde{B}_{i}), j \in \{1, ..., i\}$;
\item ($Q_{i}$): $\widetilde{B}_{i}$ satisfies the property $SD'_{i-1}(O)$;
\item ($R_{i}$): $f_{i}|_{B_{1}(\widetilde{w},\widetilde{B}_{i})} : B_{1}(\widetilde{w},\widetilde{B}_{i}) \rightarrow B_{1}(f_{i}(\widetilde{w}),X)$ is
an isomorphism onto the span of the image for
$\widetilde{w} \in \widetilde{B}_{i}$, and it is an isomorphism for $\widetilde{w} \in \widetilde{B}_{i-1}$.
\end{enumerate}

Note that ($P_{1}$), ($Q_{1}$) and ($R_{1}$) hold, i.e. that the above conditions are satisfied for $\widetilde{B}_{1}$ and $f_{1}$. We construct further $\widetilde{B}_{i+1}$ and the map
$f_{i+1} : \widetilde{B}_{i+1} \rightarrow X$. For a simplex $\widetilde{\sigma}$ of $\widetilde{B}_{i}$, we denote by $\sigma$ its image $f_{i}(\widetilde{\sigma})$ in $X$. Let
$\widetilde{S}_{i} = S_{i}(v, \widetilde{B}_{i})$ and let \begin{center}
$Z = \{ (\widetilde{w},z) \in \widetilde{S}^{(0)}_{i} \times X^{(0)} | z \in X_{w} \setminus f_{i}((\widetilde{B}_{i})_{\widetilde{w}}) \}$.
\end{center} We define a relation $\stackrel{e}{\sim}$ on $Z$ as follows: \begin{center} $(\widetilde{w}, z) \stackrel{e}{\sim} (\widetilde{w}', z')$ iff $(z = z'$ and
$\langle \widetilde{w}, \widetilde{w}'\rangle \in \widetilde{B}_{i}^{(1)})$ . \end{center} In order to define $\widetilde{B}_{i+1}$, we shall use the transitive closure $\stackrel{\overline{e}}{\sim}$
of the relation $\stackrel{e}{\sim}$. The rest of the proof relies on the following lemma.

\begin{lemma}\label{4.2}
If $(\widetilde{w}_{1}, z) \stackrel{e}{\sim} (\widetilde{w}_{2}, z) \stackrel{e}{\sim} (\widetilde{w}_{3}, z) \stackrel{e}{\sim} (\widetilde{w}_{4}, z)$ then there is
$(\widetilde{x}, z) \in Z$ such that $(\widetilde{w}_{1}, z) \stackrel{e}{\sim} (\widetilde{x}, z) \stackrel{e}{\sim} (\widetilde{w}_{4}, z)$.

\end{lemma}

\begin{proof}
By ($P_{i}$) and ($Q_{i}$), in $\widetilde{B}_{i-1}$ there are vertices $\widetilde{u}_{j}$ such that
$\langle  \widetilde{w}_{j}, \widetilde{u}_{j}, \widetilde{w}_{j+1} \rangle \in \widetilde{B}_{i-1}$, $1 \leq j \leq 3$ and there are vertices $\widetilde{u}_{j}'$ such that
$\widetilde{u}_{j}' \sim \widetilde{u}_{j}, \widetilde{w}_{j+1}, \widetilde{u}_{j+1}, 1 \leq j \leq 2$ (possibly with $\widetilde{u}_{j}' = \widetilde{u}_{j+1}, 1 \leq j \leq 2$).
Let $u_{j} = f_{i-1}(\widetilde{u}_{j}), 1 \leq j \leq 3$ and let $u'_{j} = f_{i-1}(\widetilde{u'}_{j}), 1 \leq j \leq 2$ be vertices of $X$. The ($R_{i}$) condition also implies that $u_{j} \sim w_{j}, w_{j+1}, 1 \leq j \leq 3$. Also, by the ($R_{i}$) condition, we have $u'_{j} \sim u_{j}, w_{j+1}, u_{j+1}$ (possibly with
 $u'_{j} = u_{j+1}), 1 \leq j \leq 2$.

We explain in detail why the graph induced by $w_{1}, w_{2}, w_{3}, w_{4}, u_{1}, u_{1}', u_{2}, u_{2}', u_{3}$ in $X$ is isomorphic to the graph induced by $\widetilde{w}_{1}, \widetilde{w}_{2}, \widetilde{w}_{3}, \widetilde{w}_{4}, \widetilde{u}_{1}, \widetilde{u}_{1}', \widetilde{u}_{2}, \widetilde{u}_{2}', \widetilde{u}_{3}$ in $\widetilde{B}_{i}$.
Because $\widetilde{u}_{1} \in \widetilde{B}_{i-1}$, by the ($R_{i}$) condition $f_{i}|_{B_{1}(\widetilde{u}_{1}, \widetilde{B}_{i})}$ is an isomorphism. Then since in $\widetilde{B}_{i}$ we have $\widetilde{u}_{1} \sim \widetilde{w}_{1}, \widetilde{w}_{2}, \widetilde{u}_{1}'$, in $X$ there are vertices $u_{1}, w_{1},w_{2},u_{1}'$ such that $f_{i}(\widetilde{u}_{1}) = u_{1}$, $f_{i}(\widetilde{w}_{1}) = w_{1}$, $f_{i}(\widetilde{w}_{2}) = w_{2}$, $f_{i}(\widetilde{u}_{1}') = u_{1}'$, $u_{1} \sim w_{1}, w_{2}, u_{1}'$.
Because $\widetilde{u}_{1}' \in \widetilde{B}_{i-1}$, by the ($R_{i}$) condition $f_{i}|_{B_{1}(\widetilde{u}_{1}', \widetilde{B}_{i})}$ is an isomorphism. Then since in $\widetilde{B}_{i}$ we have $\widetilde{u}_{1}' \sim \widetilde{u}_{2}, \widetilde{w}_{2}$, in $X$ there is a vertex $u_{2}$ such that $f_{i}(\widetilde{u}_{2}) = u_{2}$, $u_{1}' \sim u_{2}, w_{2}$.
Because $\widetilde{u}_{2} \in \widetilde{B}_{i-1}$, by the ($R_{i}$) condition $f_{i}|_{B_{1}(\widetilde{u}_{2}, \widetilde{B}_{i})}$ is an isomorphism. Then since in $\widetilde{B}_{i}$ we have $\widetilde{u}_{2} \sim \widetilde{w}_{2}, \widetilde{w}_{3}, \widetilde{u}_{2}'$, in $X$ there are vertices $w_{3}$, $u_{2}'$ such that $f_{i}(\widetilde{w}_{3}) = w_{3}$, $f_{i}(\widetilde{u}_{2}') = u_{2}'$, $u_{2} \sim w_{2}, w_{3}, u_{2}'$.
Because $\widetilde{u}_{2}' \in \widetilde{B}_{i-1}$, by the ($R_{i}$) condition $f_{i}|_{B_{1}(\widetilde{u}_{2}', \widetilde{B}_{i})}$ is an isomorphism. Then since in $\widetilde{B}_{i}$ we have $\widetilde{u}_{2}' \sim \widetilde{u}_{3}, \widetilde{w}_{3}$, there is a vertex $u_{3}$ in $X$ such that $f_{i}(\widetilde{u}_{3}) = u_{3}$, $u_{2}' \sim u_{3}, w_{3}$.
Because $\widetilde{u}_{3} \in \widetilde{B}_{i-1}$, by the ($R_{i}$) condition $f_{i}|_{B_{1}(\widetilde{u}_{3}, \widetilde{B}_{i})}$ is an isomorphism. Then since in $\widetilde{B}_{i}$ we have $\widetilde{u}_{3} \sim \widetilde{w}_{3}, \widetilde{w}_{4}$, in $X$ there exists a vertex $w_{4}$ such that $f_{i}(\widetilde{w}_{4}) = w_{4}$, $u_{3} \sim w_{3}, w_{4}$.

We treat first some particular cases ($a.$ - $c.$).

$a.$ We consider the case when $w_{1} = w_{4}$. Because $w_{1} \sim w_{2} \sim w_{3} \sim w_{4}$, we have $w_{1}, w_{3} \in X_{w_{2}}$. Thus $w_{1} \sim w_{3}$.
Further, since $\widetilde{w}_{1} \sim \widetilde{w}_{2} \sim \widetilde{w}_{3}$, the ($R_{i}$) condition implies that
$\widetilde{w}_{1} \sim \widetilde{w}_{3}$. By construction we have $w_{3} \sim w_{4}$. Then, by the ($R_{i}$) condition, it follows that $\widetilde{w}_{3} \sim \widetilde{w}_{4}$.
Thus, since $\widetilde{w}_{1}$ and $\widetilde{w}_{4}$ are both adjacent to $\widetilde{w}_{3}$ in $\widetilde{B}_{i}$, the lemma holds in this case trivially.

\begin{figure}[h]
    \begin{center}
        \includegraphics[height=2cm]{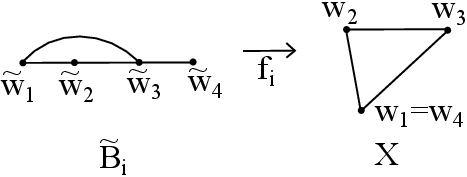}
 \caption{$w_{1} = w_{4}$}
    \end{center}
\end{figure}

$b.$ We treat the case when $w_{1} \sim w_{4}$, $\widetilde{u}_{1} = \widetilde{u}_{3}$. Since
$\widetilde{w}_{1}, \widetilde{w}_{4} \in (\widetilde{B}_{i})_{\widetilde{u}_{1}}$ and $w_{1} \sim w_{4}$, the ($R_{i}$) condition implies that $\widetilde{w}_{1} \sim \widetilde{w}_{4}$.
The lemma holds in this case trivially.

\begin{figure}[ht]
    \begin{center}
        \includegraphics[height=2cm]{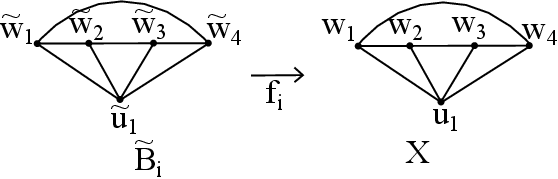}
 \caption{$w_{1} \sim w_{4}$, $\widetilde{u}_{1} = \widetilde{u}_{3}$}
    \end{center}
\end{figure}

$c.$ We consider the case when $w_{1} \sim w_{4}$, $\widetilde{u}_{1} \neq \widetilde{u}_{3}$.
Let $\widetilde{\gamma}$ be a tightening of the path $(\widetilde{w}_{1}, \widetilde{u}_{1}, \widetilde{u}_{1}', \widetilde{u}_{2}, \widetilde{u}_{2}', \widetilde{u}_{3}, \widetilde{w}_{4})$ in $\widetilde{B}_{i}$. So $\widetilde{\gamma}$ is full in $\widetilde{B}_{i}$.
We will show that the path $f_{i}(\widetilde{\gamma})$ is full in $X$.
Assume by contradiction that $f_{i}(\widetilde{\gamma})$ is not full in $X$. So there are two adjacent vertices, say $s$ and $t$, such that $\langle s,t \rangle \not \subset f_{i}(\widetilde{\gamma})$.  Note that $s,t \in \{w_{1}, u_{1}, u_{1}', u_{2}, u_{2}', u_{3}, w_{4}\}$. There are two cases to be analyzed which are treated below.

$\bullet$ Assume that there exists a vertex $v \in f_{i}(\widetilde{\gamma})$ such that $s \sim v \sim t$. Note that $v \in \{u_{1}, u_{1}', u_{2}, u_{2}', u_{3}\}$. By the ($R_{i}$) condition, there are vertices $\widetilde{s}, \widetilde{t}, \widetilde{v} \in \widetilde{B}_{i}$ such that $f_{i}(\widetilde{s}) = s$, $f_{i}(\widetilde{t}) = t$, $f_{i}(\widetilde{v}) = v$ and $\widetilde{s} \sim \widetilde{v} \sim \widetilde{t}$. Note that $\widetilde{s},\widetilde{t} \in \{\widetilde{w}_{1}, \widetilde{u}_{1}, \widetilde{u}_{1}', \widetilde{u}_{2}, \widetilde{u}_{2}', \widetilde{u}_{3}, \widetilde{w}_{4}\}$ and $\widetilde{v} \in \{\widetilde{u}_{1}, \widetilde{u}_{1}', \widetilde{u}_{2}, \widetilde{u}_{2}', \widetilde{u}_{3}\}$.
Because $s,t \in X_{v}$, the ($R_{i}$) condition implies that $\widetilde{s}, \widetilde{t} \in (\widetilde{B_{i}})_{\widetilde{v}}$. Hence $\widetilde{s} \sim \widetilde{t}$. So $\widetilde{s}$ and $\widetilde{t}$ are nonconsecutive vertices of $\widetilde{\gamma}$ which are adjacent. This implies that $\widetilde{\gamma}$ is not full in $\widetilde{B}_{i}$. Because $\widetilde{\gamma}$ is full in $\widetilde{B}_{i}$, we have reached a contradiction. The path $f_{i}(\widetilde{\gamma})$ is therefore full in $X$.
So the loop $\gamma = f_{i}(\widetilde{\gamma}) \cup \langle w_{4}, w_{1} \rangle$ is full in $X$.

 $\bullet$ Assume that there exists a sequence of vertices $(v_{j})_{1 \leq j \leq k} \in f_{i}(\widetilde{\gamma})$ such that $s \sim v_{1} \sim ... \sim v_{k} \sim t$, $2 \leq k \leq 5$.
 Note that $v_{j} \in \{u_{1}, u_{1}', u_{2}, u_{2}', u_{3}\}, 1 \leq j \leq k$. By the ($R_{i}$) condition, there are vertices $\widetilde{s}, \widetilde{t}$ and $\widetilde{v}_{j}$ in $\widetilde{B}_{i}$ such that $f_{i}(\widetilde{s}) = s$, $f_{i}(\widetilde{t}) = t$, $f_{i}(\widetilde{v_{j}}) = v_{j}$, $1 \leq j \leq k$ and such that $\widetilde{s} \sim \widetilde{v}_{1} \sim ... \sim \widetilde{v}_{k} \sim \widetilde{t}$. Note that $\widetilde{s},\widetilde{t} \in \{\widetilde{w}_{1}, \widetilde{u}_{1}, \widetilde{u}_{1}', \widetilde{u}_{2}, \widetilde{u}_{2}', \widetilde{u}_{3}, \widetilde{w}_{4}\}$ and $\widetilde{v_{j}} \in \{\widetilde{u}_{1}, \widetilde{u}_{1}', \widetilde{u}_{2}, \widetilde{u}_{2}', \widetilde{u}_{3}\}, 1 \leq j \leq k$.
Since the vertices $s$ and $t$ are adjacent, we have $s,v_{k} \in X_{t}$. Then the ($R_{i}$) condition implies that $\widetilde{s},\widetilde{v}_{k} \in (\widetilde{B_{i}})_{\widetilde{t}}$. Thus $\widetilde{s} \sim \widetilde{v}_{k}$. Because $\widetilde{s}$ and $\widetilde{v}_{k}$ are nonconsecutive vertices of $\widetilde{\gamma}$ which are adjacent, $\widetilde{\gamma}$ is not full in $\widetilde{B}_{i}$. Because $\widetilde{\gamma}$ is full in $\widetilde{B}_{i}$, we have reached a contradiction. The path $f_{i}(\widetilde{\gamma})$ is therefore full in $X$.
So the loop $\gamma = f_{i}(\widetilde{\gamma}) \cup \langle w_{4}, w_{1} \rangle$ is full in $X$.

Since $\widetilde{B}_{i}$ is simply connected, the cycle $\widetilde{\gamma} \cup \langle \widetilde{w}_{4}, \widetilde{w}_{3} \rangle \cup \langle \widetilde{w}_{3}, \widetilde{w}_{2} \rangle \cup \langle \widetilde{w}_{2}, \widetilde{w}_{1} \rangle$ in $\widetilde{B}_{i}$ is contractible and it has therefore a filling.
The loop $f_{i}(\widetilde{\gamma}) \cup \langle w_{4}, w_{3} \rangle \cup \langle w_{3}, w_{2} \rangle \cup \langle w_{2}, w_{1} \rangle$ in $X$ is the image through the simplicial map $f_{i}$ of the loop $\widetilde{\gamma} \cup \langle \widetilde{w}_{4}, \widetilde{w}_{3} \rangle \cup \langle \widetilde{w}_{3}, \widetilde{w}_{2} \rangle \cup \langle \widetilde{w}_{2}, \widetilde{w}_{1} \rangle$ in $\widetilde{B}_{i}$.  So the loop $f_{i}(\widetilde{\gamma}) \cup \langle w_{4}, w_{3} \rangle \cup \langle w_{3}, w_{2} \rangle \cup \langle w_{2}, w_{1} \rangle$ also has a filling. Because $w_{1} \sim w_{4}$ the loop $(w_{1}, w_{2}, w_{3}, w_{4})$ in $X$ is homotopically trivial and hence it has a filling.
Thus there is a filling of $\gamma = f_{i}(\widetilde{\gamma}) \cup \langle w_{4}, w_{1} \rangle$ which is the union of the fillings of the loops $f_{i}(\widetilde{\gamma}) \cup \langle w_{4}, w_{3} \rangle \cup \langle w_{3}, w_{2} \rangle \cup \langle w_{2}, w_{1} \rangle$ and $(w_{1}, w_{2}, w_{3}, w_{4})$.
Hence the loop $\gamma$ in $X$ is homotopically trivial.
Because $\gamma$ is full and its length is at most $7$, by $8$-location, there exists a vertex $y$ in $X$ which is adjacent to all vertices of $\gamma$. Thus $\langle  y,u_{1} \rangle$ and $\langle  y,u_{3} \rangle$ are edges of $X$. By the ($R_{i}$) condition applied to the vertices $\widetilde{u}_{1}$ and $\widetilde{u}_{3}$, there is a vertex $\widetilde{y}$ in $\widetilde{B}_{i}$ such that
$f_{i}(\widetilde{y}) = y$ while
$\langle  \widetilde{y}, \widetilde{u}_{1} \rangle$ and $\langle  \widetilde{y}, \widetilde{u}_{3} \rangle$ are edges of $\widetilde{B}_{i}$.
Then, because $y \sim w_{4}$, the ($R_{i}$) condition implies that $\widetilde{y} \sim \widetilde{w}_{4}$.
 Similarly we get $\widetilde{y} \sim \widetilde{w}_{1}$.
Because $\widetilde{w}_{1}$ and $\widetilde{w}_{4}$
both belong to the link of $\widetilde{y}$ in $\widetilde{B}_{i}$ and $w_{1} \sim w_{4}$, the ($R_{i}$) condition implies that $\widetilde{w}_{1} \sim \widetilde{w}_{4}$. So the lemma holds in the case $w_{1} \sim w_{4}$, $\widetilde{u}_{1} \neq \widetilde{u}_{3}$ trivially.

\begin{figure}[ht]
    \begin{center}
        \includegraphics[height=4cm]{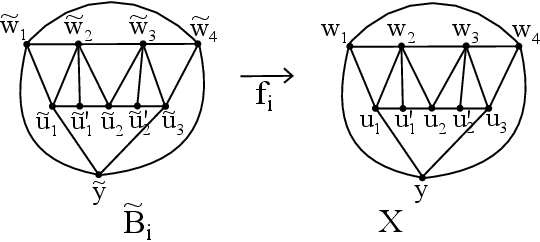}
 \caption{$w_{1} \sim w_{4}$, $d_{c}(\widetilde{u}_{1}, \widetilde{u}_{2}) = d_{c}(\widetilde{u}_{2}, \widetilde{u}_{3}) = 2$}
    \end{center}
\end{figure}

\begin{figure}[h]
    \begin{center}
        \includegraphics[height=3cm]{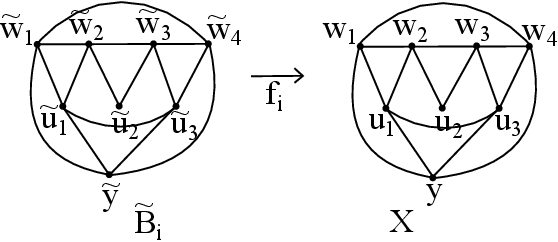}
 \caption{$w_{1} \sim w_{4}$, $\widetilde{u}_{1} \sim \widetilde{u}_{3}$}
    \end{center}
\end{figure}

We show further that $z \neq u_{1}$. Suppose by contradiction that $z = u_{1}$. By the definition of the set $Z$, we have $z \in X_{w_{1}}$.
Since $z = u_{1}$ and $u_{1} \sim w_{1}$, by the ($R_{i}$) condition applied to the vertex $\widetilde{w}_{1}$, there is a vertex $\widetilde{z}$ in $\widetilde{B}_{i-1}$
such that $f_{i}(\widetilde{z}) = z$ and $\widetilde{z} \sim \widetilde{w}_{1}$. So $\langle  \widetilde{z}, \widetilde{w}_{1} \rangle \in \widetilde{B}_{i}$.
But since $(\widetilde{w}_{1}, z) \in Z$, it follows that $\langle  \widetilde{z}, \widetilde{w}_{1} \rangle \notin \widetilde{B}_{i}$. Because we have reached a contradiction, $z \neq u_{1}$. One can show similarly that $z \neq u_{i}, i \in \{2,3\}$ and $z \neq u'_{j}, j \in \{1,2\}$.

We prove next that $z \nsim u_{1}$. Suppose by contradiction that $z \sim u_{1}$. By the definition of the set $Z$, we have $z \in X_{w_{1}}$. By the ($R_{i}$) condition applied to the vertex $\widetilde{u}_{1}$, there is a vertex $\widetilde{z}$ in $\widetilde{B}_{i-1}$
 such that $f_{i}(\widetilde{z}) = z$ and $\widetilde{z} \sim \widetilde{u}_{1}$.
Because $\widetilde{w}_{1}$ and $\widetilde{z}$ both belong to the link of $\widetilde{u}_{1}$ in $\widetilde{B}_{i}$, we have
$\langle  \widetilde{z}, \widetilde{w}_{1} \rangle \in \widetilde{B}_{i}$.
But since $(\widetilde{w}_{1}, z) \in Z$, it follows that $\langle  \widetilde{z}, \widetilde{w}_{1} \rangle \notin \widetilde{B}_{i}$.
This yields a contradiction and hence $z \nsim u_{1}$. One can show similarly that $z \nsim u_{i}, i \in \{2,3\}$ and $z \nsim u'_{j}, j \in \{1,2\}$.

From now on assume that $w_{1} \neq w_{4}$, $w_{1} \nsim w_{4}$, $z \neq u_{i}, z \nsim u_{i}, 1 \leq i \leq 3$,
$z \neq u_{j}', z \nsim u_{j}', j \in \{1,2\}$.

Let $\widetilde{\gamma}$ be a tightening of the path $(\widetilde{w}_{1}, \widetilde{u}_{1}, \widetilde{u}'_{1}, \widetilde{u}_{2}, \widetilde{u}'_{2}, \widetilde{u}_{3}, \widetilde{w}_{4})$ in $\widetilde{B_{i}}$. Since the path $\widetilde{\gamma}$ is full in $\widetilde{B_{i}}$,
the ($R_{i}$) condition implies that the path $f_{i}(\widetilde{\gamma})$ is full in $X$.
Moreover, by the definition of the set $Z$, the loop $\gamma = f_{i}(\widetilde{\gamma}) \cup \langle w_{4}, z \rangle \cup \langle z, w_{1} \rangle$ in $X$ is
homotopically trivial (if $\widetilde{u}_{1} \sim \widetilde{u}_{3}$ the claim requires additional explanations given in the paragraph below; assuming the claim, we continue).
Because $\gamma$ is full and its length is at most $8$, by $8$-location, $\gamma$ is contained
in the link of a vertex $x$.
By the ($R_{i}$) condition applied to the vertex $\widetilde{u}_{1}$, there exists a vertex $\widetilde{x}$ in $\widetilde{B}_{i}$ such that
$f_{i}(\widetilde{x}) = x$ and
$\widetilde{x} \sim \widetilde{u}_{1}$.
Moreover, since $w_{1} \sim x$ and $w_{4} \sim x$, the ($R_{i}$) condition implies that the vertices
$\widetilde{w}_{4}$ and $\widetilde{w}_{1}$ are both adjacent to
$\widetilde{x}$.
Hence, once we show that $(\widetilde{x}, z) \in Z$, the lemma is proven.
Assume by contradiction that $(\widetilde{x}, z) \notin Z$. By the ($R_{i}$) condition, there exists a vertex $\widetilde{z} \in \widetilde{B_{i}}$ such that $f_{i}(\widetilde{z}) = z$ and
$\langle \widetilde{z}, \widetilde{x} \rangle \in \widetilde{B_{i}}$. Because $\widetilde{w}_{1}$ and $\widetilde{z}$ both belong to the link of $\widetilde{x}$ in
$\widetilde{B}_{i}$, this implies that $\langle \widetilde{z}, \widetilde{w}_{1} \rangle \in \widetilde{B}_{i}$. But, since $(\widetilde{w}_{1}, z) \in Z$, we have
$\langle \widetilde{z}, \widetilde{w}_{1}\rangle \notin \widetilde{B_{i}}$. Because we have reached a contradiction, it follows that $(\widetilde{x},z) \in Z$.

For the case $\widetilde{u}_{1} \sim \widetilde{u}_{3}$ we explain further why the loop $\gamma = (z, w_{1}, u_{1}, u_{3}, w_{4})$ in $X$ is homotopically trivial.
Note that the filling of $\gamma$ is the union of the fillings of the loops $(z, w_{1}, w_{2}, w_{3}, w_{4})$ (which is homotopically trivial) and $(w_{1}, w_{2}, w_{3}, w_{4}, u_{3}, u_{1})$.
The loop $(w_{1}, w_{2}, w_{3}, w_{4}, u_{3}, u_{1})$ is the image through the simplicial map $f_{i}$ of the loop $(\widetilde{w}_{1}, \widetilde{w}_{2}, \widetilde{w}_{3}, \widetilde{w}_{4}, \widetilde{u}_{3}, \widetilde{u}_{1})$ in $\widetilde{B}_{i}$. Because $\widetilde{B}_{i}$ is simply connected, the loop $(\widetilde{w}_{1}, \widetilde{w}_{2}, \widetilde{w}_{3}, \widetilde{w}_{4}, \widetilde{u}_{3}, \widetilde{u}_{1})$
is contractible and it has therefore a filling.
So the loop $(w_{1}, w_{2}, w_{3}, w_{4}, u_{3}, u_{1})$ in $X$ also has a filling. Hence this loop is homotopically trivial. So $\gamma$ is homotopically trivial.

\begin{figure}[ht]
    \begin{center}
        \includegraphics[height=3cm]{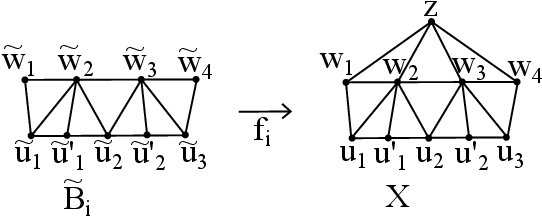}
 \caption{$d_{c}(\widetilde{u}_{1}, \widetilde{u}_{2}) = d_{c}(\widetilde{u}_{2}, \widetilde{u}_{3}) = 2$}
    \end{center}
\end{figure}

\begin{figure}[ht]
    \begin{center}
        \includegraphics[height=3cm]{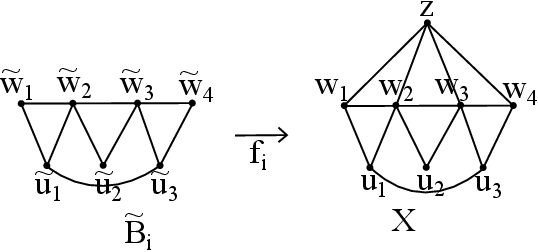}
 \caption{$\widetilde{u}_{1} \sim \widetilde{u}_{3}$}
    \end{center}
\end{figure}

\begin{figure}[ht]
    \begin{center}
        \includegraphics[height=3cm]{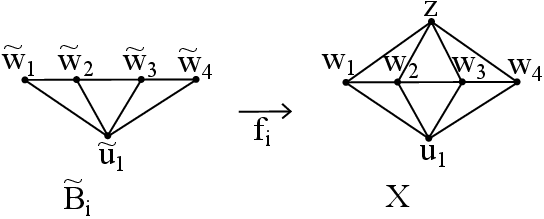}
 \caption{$\widetilde{u}_{1} = \widetilde{u}_{3}$}
    \end{center}
\end{figure}

\end{proof}

According to the previous lemma, if $(\widetilde{u},z) \stackrel{\overline{e}}{\sim} (\widetilde{w},z)$, then there is a vertex $\widetilde{x} \in \widetilde{S}_{i}$ such that
$(\widetilde{x},z) \in Z$ and $\langle \widetilde{x},\widetilde{u}\rangle, \langle \widetilde{x},\widetilde{w}\rangle \in \widetilde{B}_{i}$.

We define the flag simplicial complex $\widetilde{B}_{i+1}$. Its $0$-skeleton is defined as the set
$\widetilde{B}_{i+1}^{(0)} = \widetilde{B}_{i}^{(0)} \cup (Z/_{\stackrel{\overline{e}}{\sim}})$.
Further we define the $1$-skeleton $\widetilde{B}_{i+1}^{(1)}$ of $\widetilde{B}_{i+1}$. Edges between vertices of $\widetilde{B}_{i}$ are the same as in $\widetilde{B}_{i}$.
For every $\widetilde{w} \in \widetilde{S}_{i}^{(0)}$, there are edges joining $\widetilde{w}$ with $[\widetilde{w},z] \in Z/_{\stackrel{\overline{e}}{\sim}}$
(here $[\widetilde{w},z]$ denotes the equivalence class of $(\widetilde{w},z)\in Z$) and there are edges joining $[  \widetilde{w},z ]  $ with $[  \widetilde{w},z' ]  $, for $\langle   z,z' \rangle \in X$.
Once we have defined the $1$-skeleton of $\widetilde{B}_{i+1}$, the higher dimensional skeleta are determined by the flagness property.

The map $f_{i+1} : \widetilde{B}_{i+1}^{(0)} \rightarrow X$ is defined by $f_{i+1}|_{\widetilde{B}_{i}} = f_{i}$ and $f_{i+1}([  \widetilde{w},z ]  ) = z$. We show that this map can be extended simplicially. The proof is analogous to the one given in \cite{O-8loc} (Lemma $3.5$). It is enough to do it for simplices in $\widetilde{B}_{i+1} \setminus \widetilde{B}_{i-1}$. Let $\widetilde{\sigma} = <[\widetilde{w}_{1},z_{1}], ..., [\widetilde{w}_{l},z_{l}],\widetilde{w}_{1}', ..., \widetilde{w}_{m}'> \in \widetilde{B}_{i+1}$ be a simplex. Then, by the definition of edges in $\widetilde{B}_{i+1}$, we have that $<z_{p},z_{q}> \in X$ and $<z_{r},w_{s}'> \in X,$ for $p,q,r \in \{ 1,2, ..., l \}$ and $s \in \{ 1,2, ..., m \}$. Since $f_{i+1}([\widetilde{w}_{p},z_{p}]) = z_{p}$, $f_{i+1}(\widetilde{w}_{s}') = w_{s}'$ and since the map $f_{i}$ is simplicial, it follows that $<\{ f_{i+1}(\widetilde{w}) | \widetilde{w} \in \widetilde{\sigma} \}> \in X$. Hence, by the simplicial extension, we can define the map $f_{i+1} : \widetilde{B}_{i+1} \rightarrow X$.

We verify next whether $\widetilde{B}_{i+1}$ and the map $f_{i+1}$ satisfy the conditions ($P_{i+1}$), ($Q_{i+1}$) and ($R_{i+1}$). The proofs of the conditions ($P_{i+1}$) and ($R_{i+1}$) are similar to the ones given in \cite{O-sdn} (Theorem $4.5$). The proof of the condition ($Q_{i+1}$) is similar to the one given in \cite{O-8loc} (Lemma $3.5$).

Condition ($P_{i+1}$). Because for every $[\widetilde{w},z] \in \widetilde{B}_{i+1}$ we have $d(v,[\widetilde{w},z]) = i+1$, it is clear that $\widetilde{B}_{j} = B_{j}(v,\widetilde{B}_{i+1})$, for $3 \leq j \leq i+1$. Thus ($P_{i+1}$) holds.

Condition ($Q_{i+1}$). According to the condition ($Q_{i}$), it suffices to check the triangle condition (T) and the vertex condition (V) from Definition $2.3$ only for, respectively, edges and vertices in $\widetilde{S}_{i+1}$. The definition of edges in $\widetilde{S}_{i+1}$ implies that $(\widetilde{B}_{i+1})_{e} \cap \widetilde{B}_{i}$ is non-empty for an edge $e \in \widetilde{S}_{i+1}$. The triangle condition (T) is hence fulfilled. The definition of edges in $\widetilde{B}_{i+1}$ and Lemma $4.2$ imply the vertex condition (V) for a vertex $[\widetilde{w},z] \in \widetilde{S}_{i+1}$.

Condition ($R_{i+1}$). Note that since ($R_{i}$) holds in $B_{i}$, it is enough to consider only vertices $\widetilde{w} \in \widetilde{B}_{i+1} \setminus \widetilde{B}_{i-1}$. There are two cases to consider.

(Case $1$. $\widetilde{w} \in \widetilde{S}_{i}$.)  First we prove injectivity of the map $f_{i+1}|_{B_{1}(\widetilde{w},\widetilde{B}_{i+1})}$. Let $\widetilde{x} \neq \widetilde{x}' \in (\widetilde{B}_{i+1})_{\widetilde{w}}$ be two vertices. If $\widetilde{x}, \widetilde{x}' \in \widetilde{S}_{i}$, then $f_{i+1}(\widetilde{x}) \neq f_{i+1}(\widetilde{x}') \neq f_{i+1}(\widetilde{w}) \neq f_{i+1}(\widetilde{x})$ by local injectivity of $f_{i}$. If $\widetilde{x} \in \widetilde{S}_{i}$ and $\widetilde{x}' = [\widetilde{w},z]$ then $f_{i+1}(\widetilde{x}) \neq f_{i+1}(\widetilde{x}') = z \neq f_{i+1}(\widetilde{w}) \neq f_{i+1}(\widetilde{x})$, by local injectivity of $f_{i}$ and by the definition of the set $Z$ containing $z$. If $\widetilde{x} = [\widetilde{w},z]$ and $\widetilde{x}' = [\widetilde{w},z']$ then $f_{i+1}(\widetilde{x}) = z \neq f_{i+1}(\widetilde{x}') = z' \neq f_{i+1}(\widetilde{w}) \neq f_{i+1}(\widetilde{x}) = z$ by the definition of the set $Z$ and by the fact that $z \neq z'$.

Surjectivity of the map $f_{i+1}|_{B_{1}(\widetilde{w},\widetilde{B}_{i+1})}$ follows from the fact that, for $z \in X_{w} \setminus f_{i}(B_{1}(\widetilde{w},\widetilde{B}_{i}))$, we have $f_{i+1}([\widetilde{w},z]) = z$.

(Case $2$. $\widetilde{w} = [\widetilde{y},z]$.) Here we prove only injectivity of the map $f_{i+1}|_{B_{1}(\widetilde{w},\widetilde{B}_{i+1})}$. Let $\widetilde{x} \neq \widetilde{x}' \in (\widetilde{B}_{i+1})_{\widetilde{w}}$ be two vertices. If $\widetilde{x}, \widetilde{x}' \in \widetilde{S}_{i}$, then $f_{i+1}(\widetilde{x}) \neq f_{i+1}(\widetilde{x}') \neq f_{i+1}(\widetilde{w}) = z \neq f_{i+1}(\widetilde{x}) $, by local injectivity of $f_{i}$ and by the definition of the set $Z$. If $\widetilde{x} \in \widetilde{S}_{i}$ and $\widetilde{x}' = [\widetilde{y}',z']$ then $f_{i+1}(\widetilde{x}) \neq f_{i+1}(\widetilde{x}') = z' \neq f_{i+1}(\widetilde{w}) = z \neq f_{i+1}(\widetilde{x})$, by the definition of the set $Z$ and since $z \neq z'$. If $\widetilde{x} = [\widetilde{y}',z']$ and $\widetilde{x}' = [\widetilde{y}'',z'']$ then $d(\widetilde{y}',\widetilde{y}'') \leq 1$ and thus $z' \neq z''$. Because $z' \neq z \neq z''$ we have $f_{i+1}(\widetilde{x}) = z' \neq f_{i+1}(\widetilde{x}') = z'' \neq f_{i+1}(\widetilde{w}) = z \neq f_{i+1}(\widetilde{x}) = z'$.

So we have built inductively a complex $\widetilde{X} = \cup_{i=1}^{\infty}\widetilde{B}_{i}$ which satisfies the $SD'_{n}(O)$ property for each natural $n$. Inductively we have also
constructed a map
$f = \cup_{i=1}^{\infty}f_{i} : \widetilde{X} \rightarrow X$ which is a covering map. Because
$\widetilde{X}$ was built such that it satisfies, for each $n$, the $SD'_{n}(O)$ property, it is, by Proposition \ref{2.1}, simply connected. So $\widetilde{X}$ is the
universal cover of $X$. Because the universal cover of $X$ is unique and the vertex $O$ is arbitrary, $\widetilde{X}$
satisfies the $SD'_{n}(O)$ property for each $O$ and for each $n$.
Hence we have constructed the universal cover of $X$ which satisfies the $SD'$ property.
\end{proof}

Theorem \ref{3.7} and Theorem \ref{4.1} imply the paper's main result.

\begin{theorem}\label{4.2}
Let $X$ be a simply connected, $8$-located simplicial complex. Then the $0$-skeleton of $X$ with a path metric induced from $X^{(1)}$, is $\delta$-hyperbolic, for a
universal constant $\delta$.
\end{theorem}

\proof[Acknowledgements]
I am indebted to Professor Damian Osajda for introducing me to the subject, posing the problem, and helpful explanations.
I thank Professor Louis Funar and Professor Dorin Andrica for introducing me to systolic geometry a few years ago. This work was partially supported by the grant $346300$ for IMPAN from
the Simons Foundation and the matching $2015-2019$ Polish MNiSW fund.

\end{document}